\documentclass[a4paper]{article}

\usepackage{amssymb}
\usepackage{amsthm}
\usepackage{amsmath}
\usepackage[affil-it]{authblk}
\usepackage{fullpage}

\newtheorem{theorem}{Theorem}
\newtheorem{proposition}{Proposition}
\newtheorem{lemma}{Lemma}
\theoremstyle{definition}
\newtheorem{remark}{Remark}

\newcommand{\EE}{\mathbb{E}}
\newcommand{\PP}{\mathbb{P}}
\newcommand{\z}{\mathbf{0}}
\newcommand{\el}{\mathbf{l}}
\newcommand{\x}{\mathbf{x}}
\newcommand{\y}{\mathbf{y}}
\newcommand{\n}{\mathbf{n}}
\newcommand{\ef}{\mathbf{f}}
\newcommand{\er}{\mathbf{r}}

\newcommand{\m}{\mathbf{m}}
\newcommand{\M}{\mathbf{M}}
\newcommand{\ES}{\mathbf{S}}
\newcommand{\ka}{\mathbf{k}}
\newcommand{\ba}{\mathbf{a}}
\newcommand{\bb}{\mathbf{b}}
\newcommand{\bc}{\mathbf{c}}
\newcommand{\q}{\mathbf{q}}
\newcommand{\bz}{\mathbf{z}}
\newcommand{\bi}{\mathbf{i}}
\newcommand{\bu}{\mathbf{u}}
\newcommand{\bv}{\mathbf{v}}

\newcommand{\X}{\mathbf{X}}
\newcommand{\N}{\mathbf{N}}
\newcommand{\Net}{\mathbb{N}^*}
\newcommand{\1}{\mathbf{1}}
\newcommand{\NN}{\mathbb{N}^{d}}
\newcommand{\RR}{ \mathbb{R}^{d}}
\newcommand{\ii}{i=1\ldots d}
\newcommand{\lf}{\left\lfloor}   
\newcommand{\rf}{\right\rfloor}

\begin{document}

\title{Beyond the $Q$-process: various ways of conditioning the multitype Galton-Watson process}
\author{Sophie P\'{e}nisson
  \thanks{\texttt{sophie.penisson@u-pec.fr}}}
\affil{Universit\'e Paris-Est, LAMA (UMR 8050), UPEMLV, UPEC, CNRS, 94010 Cr\'{e}teil, France}
\date{ }

\maketitle

\begin{abstract}
Conditioning a multitype Galton-Watson process to stay alive into the indefinite future leads to what is known as its associated $Q$-process. We show that the same holds true if the process is conditioned to reach a positive threshold or a non-absorbing state. We also demonstrate that the stationary measure of the $Q$-process, obtained by construction as two successive limits (first by delaying the extinction in the original process and next by considering the long-time behavior of the obtained $Q$-process), is as a matter of fact a double limit. Finally, we prove that conditioning a multitype branching process on having an infinite total progeny leads to a process presenting the features of a $Q$-process. It does not however coincide with the original associated $Q$-process, except in the critical regime.
\end{abstract}



\section{Introduction}

The benchmark of our study  is the $Q$-process associated with a multitype Galton-Watson (GW) process, obtained by conditioning the branching process $\X_k$ on not being extinct in the distant future ($\{\X_{k+n}\neq\z\}$, with $n\to +\infty$) and on the event that extinction takes place ($\{\lim_l \X_l=\z\}$) (see \cite{Naka78}). Our goal is to investigate some seemingly comparable conditioning results and to relate them to the $Q$-process.

After a description of the basic assumptions on the multitype GW process,  we start in Subsection \ref{sec:asso} by describing the "associated" branching process, which will be a key tool when conditioning on the event that extinction takes place, or when conditioning on an infinite total progeny. 

We shall first prove  in Section \ref{sec:threshold} that by replacing in what precedes the conditioning event  $\{\X_{k+n}\neq\z\}$ by $\{\X_{k+n}\in S\}$, where $S$ is a subset which does not contain $\z$, the obtained limit process remains the $Q$-process.  This means in particular that conditioning  in the distant future on reaching a non-zero state or a positive threshold, instead of conditioning on non-extinction, does not alter the result. 

In a second instance, we focus in the noncritical case on the stationary measure of the positive recurrent $Q$-process. Formulated in a loose manner, this measure is obtained by considering $\{\X_k\mid \X_{k+n}\neq\z\}$, by  delaying the extinction time ($n\to\infty$), and by studying the long-time behavior of the limit process ($k\to\infty$). It is already known (\cite{Naka78}) that inverting the limits leads to the same result. We prove in Section \ref{sec:Yaglom double} that  the convergence to the stationary measure still holds even if $n$ and $k$ simultaneously grow to infinity. This requires an additional second-order moment assumption if the process is subcritical.

Finally, we investigate in Section \ref{sec:totalprog} the distribution of the multitype GW process conditioned on having an infinite total progeny. This is motivated by Kennedy's result, who studies in \cite{Ken75} the behavior of a monotype GW process $ X_k$ conditioned on the event $\{N = n\}$ as $n\to+\infty$, where $N=\sum_{k=0}^{+\infty}X_k$ denotes the total progeny. Note that the latter conditioning seems comparable to the device of conditioning  on the event that extinction occurs but has not done so by generation $n$. It is indeed proven in the aforementioned paper that in the critical case, conditioning on the total progeny or on non-extinction indifferently results in the $Q$-process. This result has since then been extended for instance to monotype GW trees and to other conditionings: in the critical case, conditioning a GW tree by its height, by its total progeny  or by its number of leaves leads to the same limiting tree (see e.g. \cite{AbrDel14,Jan12}). However, in the noncritical case, the two methods provide different limiting results: the limit process is always the $Q$-process of some critical process, no matter the class of criticality of the original process. Under a moment assumption (depending on the number of types of the process), we generalize this result  to the multitype case. For this purpose we assume that the total progeny increases to infinity according to the "typical" limiting type proportions of the associated critical GW process, by conditioning on the event $\{\N = \lf n\mathbf{w}\rf\}$ as $n\to\infty$, where $\mathbf{w}$ is a left eigenvector related to the maximal eigenvalue 1 of the mean  matrix of the critical process.

\subsection{Notation}
\label{sec:notation}
Let $d\geqslant 1$. In this paper, a generic point in $\RR$ is denoted by $\x=(x_1,\ldots,x_d)$, and its transpose is written $\x^T$. By $\mathbf{e}_i=(\delta_{i,j}) _{1\leqslant j\leqslant d}$ we denote the $i$-th unit vector in $\RR$, where $\delta_{i,j}$ stands for the Kronecker
delta. We write $\z=\left( 0,\ldots,0\right) $ and $\1=\left( 1,\ldots,1\right) $. The notation $\x\y$ (resp.  $\lf \x \rf$) stands for the vector with coordinates $x_iy_i$  (resp. $\lf x_i \rf$,  the integer part of $x_i$). We denote by $\x^{\y}$ the product $\prod_{i=1}^dx_i^{y_i}$. The obvious partial order on $\RR$ is  $\x\leqslant \y$, when  $x_i\leqslant y_i$ for each $i$, and $\x<\y$  when  $x_i< y_i$ for each $i$. Finally,  $\x \cdot \y$ denotes the scalar product in $\RR$, $\|\x\|_1$ the $L^1$-norm and $\|\x\|_2$ the $L^2$-norm.
 
\subsection{Multitype GW processes}
 Let $( \X_k)_{k\geqslant 0}$ denote a $d$-type GW process, with $n$-th transition probabilities $P_n\left( \x,\y\right)= \PP( \X_{k+n}=\y\mid\X_{k}=\x)$, $ k$, $n\in\mathbb{N}$, $\x$, $\y\in\NN$. Let $\ef=\left( f_1,\ldots,f_d\right) $ be its offspring generating function, where for each $i=1\ldots d$ and $\er\in[0,1]^d$, $f_i\left( \er\right) =\EE_{\mathbf{e}_i}(\er^{\X_1})=\sum_{\ka\in\NN}p_i\left( \ka\right) \er^{\ka}$, the subscript $\mathbf{e}_i$ denoting the initial condition, and $p_i$ the offspring probability distribution of type $i$. For each $i$, we denote by $\m^i=\left( m_{i1},\ldots,m_{id}\right) $ (resp. $\mathbf{\Sigma}^{i}$) the mean vector (resp. covariance matrix) of the offspring probability distribution $p_i$. The mean matrix is then given by $\mathbf{M}=(m_{ij})_{1\leqslant i,j \leqslant d}$. If it exists, we denote by $\rho$ its Perron's root, and by $\bu$ and $\bv$ the associated right and left eigenvectors (i.e. such that $\mathbf{M}\bu^T=\rho\bu^T$, $\bv\mathbf{M}=\rho\bv$), with the normalization convention $\bu\cdot \1=\bu\cdot\bv=1$. The process is then called critical (resp. subcritical, supercritical) if $\rho=1$ (resp. $\rho<1$, $\rho>1$). In what follows we shall denote by $\ef_n$ the  $n$-th iterate of the function $\ef$, and by $\mathbf{M}^n=(m_{ij}^{(n)})_{1\leqslant i,j \leqslant d}$ the $n$-th power of the matrix $\M$, which correspond respectively to the generating function and mean matrix of the process at time $n$. By the branching property, for each $\x\in\NN$, the function $\ef_n^\x$ then corresponds to the generating function of the process at time $n$ with initial state $\x$, namely $\EE_{\x}(\er^{\X_n})=\ef_n\left( \er\right)^\x$. Finally, we define the extinction time  $T=\inf\{k\in\mathbb{N},\,\X_k=\z\}$, and the extinction probability  vector  $\q=(q_1,\ldots,q_d)$, given by $q_i=\PP_{\mathbf{e}_i}\left( T<+\infty\right)$, $i=1\ldots d$.

\subsection{Basic assumptions}
\label{sec:basic assumptions}
\begin{enumerate}
\item[$(A_1)$] The mean matrix $\mathbf{M}$ is finite. The process is nonsingular ($\ef(\er)\neq \mathbf{M}\er$), is positive regular (there exists some $n\in\Net$ such that each entry of $\mathbf{M}^n$ is positive), and is such that $\mathbf{q}>\z$.
\end{enumerate}
The latter statement will always be assumed. It ensures in particular the existence of the Perron's root $\rho$ and that (\cite{Karl66}),
\begin{equation}\label{mean}
\lim_{n\to+\infty}\rho^{-n}m^{(n)}_{ij}=u_iv_j.
\end{equation} When necessary, the following additional assumptions will be made.
\begin{enumerate}
\item[$(A_2)$]  For each $i,j=1\ldots d$, $\mathbb{E}_{\mathbf{e}_i}( X_{1,j}\ln X_{1,j})<+\infty$.
\item[$(A_3)$]  The covariance matrices $\mathbf{\Sigma}^i$, $i=1\ldots d$, are finite.
\end{enumerate}

\subsection{The associated process}
\label{sec:asso}
For any vector $\ba>\z$ such that for each $i=1\ldots d$, $f_i(\ba)<+\infty$, we define the generating function $\overline{\ef}=\left( \overline{f}_1,\ldots,\overline{f}_d\right) $ on $[0,1]^d$ as follows: \[\overline{f}_i\left( \er\right)=\frac{f_i\left( \ba \er\right)}{f_i\left( \ba\right)},\ \ i=1\ldots d.\]
We then denote by $\big(\overline{\X}_k\big)_{k\geqslant 0}$ the GW process with offspring generating function $\overline{\ef}$, which will be referred to as the \textit{associated process} with respect to $\ba$. We shall denote by $\overline{P}_n$, $\overline{p}_i$ etc. its transition probabilities, offspring probability distributions etc. We easily compute that for each $n\geqslant 1$, $i=1\ldots d$, $\ka\in\NN$ and $\er\in[0,1]^d$, denoting by $*$ the convolution product,
\begin{align}\label{off}
\overline{p}_i^{*n}\left(\ka\right)=\frac{\ba^{\ka}}{f_i\left( \ba\right)^{n} }p_i^{*n}\left(\ka\right),\ \ \  
\overline{f}_{n,i}\left(\er\right)=\frac{f_{n,i}\left(\ba\er\right)}{f_i\left( \ba\right)^{n} }.
\end{align}
\begin{remark}\label{rem: subcritic} It is known (\cite{JagLag08}) that a supercritical GW process conditioned on the event $\{T<+\infty\}$ is subcritical. By construction, its  offspring generating function is given by $\er\mapsto f_i(\mathbf{q}\er)/q_i$. Since the extinction probability vector satisfies $\ef(\mathbf{q})=\mathbf{q}$ (\cite{Har63}), this means that the associated process $\big(\overline{\X}_k\big)_{k\geqslant 0}$ with respect to $\mathbf{q}$ is subcritical.
\end{remark}

\section{Classical results: conditioning on non-extinction}
\label{sec:nonext}

\subsection{The Yaglom distribution (\cite{JofSpit67}, Theorem 3)}
\label{sec:Yaglom}
Let $\left(\X_k\right)_{k\geqslant 0} $ be a subcritical multitype GW process satisfying $(A_1)$. Then for all $\x_0,\bz\in\NN\setminus\{\z\}$, 
\begin{equation}\label{Yaglom}
\lim_{k\to+\infty}\PP_{\x_0}\left( \X_{k}=\bz \mid \X_k\neq \z\right)=\nu(\bz),
\end{equation}
where $\nu$ is a probability distribution on $\NN\setminus\{\z\}$ independent of the initial state $\x_0$. This quasi-stationary distribution is often referred to as the Yaglom distribution associated with  $\left(\X_k\right)_{k\geqslant 0}$. We shall denote by $g$ its generating function $g(\er)=\sum_{\bz\neq \z}\nu(\bz)\er^{\bz}$. Under $(A_2)$, $\nu$ admits finite and positive first moments
\begin{equation}\label{first moment g}
\frac{\partial g \left( \mathbf{1}\right) }{\partial r_i}=v_i\gamma^{-1},\ \ \ii,
\end{equation}
where $\gamma>0$ is a limiting quantity satisfying for each $\x\in\NN\setminus\{\z\}$,
\begin{equation}\label{againbasic}
\lim_{k\to +\infty}\rho^{-k}\PP_{\x}\left( \X_k\neq\z\right) = \gamma\,\x\cdot\bu.
\end{equation}

\subsection{The $Q$-process (\cite{Naka78}, Theorem 2)}
\label{sec:Q process}
Let $\left(\X_k\right)_{k\geqslant 0} $ be a multitype GW process satisfying $(A_1)$. Then for all $\x_0\in\NN\setminus\{\z\}$, $k_1\leqslant\ldots\leqslant k_j\in\mathbb{N}$,  and $\x_1,\ldots,\x_j\in\NN$,
\begin{multline}\label{limext}
\lim_{n\to+\infty}\PP_{\x_0}\left( \X_{k_1}=\x_1,\ldots, \X_{k_j}=\x_j\mid \X_{k_j+n}\neq\z,\,T<+\infty\right)\\=\frac{1}{\overline{\rho}^{k_j}}\frac{\x_j\cdot\overline{\bu}}{\x_0\cdot\overline{\bu}}\PP_{\x_0}\left( \overline{\X}_{k_1}=\x_1,\ldots, \overline{\X}_{k_j}=\x_j\right),
\end{multline}
where $\big(\overline{\X}_k\big)_{k\geqslant 0}$ is  the associated process with respect to $\q$. As told in the introduction, this limiting process is the $Q$-process associated with $\left(\X_k\right)_{k\geqslant 0}$. It is Markovian with transition probabilities
\begin{equation*}
Q_1\left( \x,\y\right)=\frac{1}{\overline{\rho}}\frac{\y\cdot\overline{\bu}}{\x\cdot\overline{\bu}}\overline{P}_1\left( \x,\y\right)=\frac{1}{\overline{\rho}}\q^{\y-\x}\frac{\y\cdot\overline{\bu}}{\x\cdot\overline{\bu}}P_1\left( \x,\y\right),\ \ \ \ \ \ \x,\y\in\NN\setminus\{\z\}.
\end{equation*}
If $\rho>1$, the $Q$-process is positive recurrent. If  $\rho=1$, it is transient. If   $\rho<1$, the $Q$-process is positive recurrent if and only if $(A_2)$ is satisfied. In the positive recurrent case, the stationary measure for the $Q$-process is given by the size-biased Yaglom distribution
\begin{equation}\label{size biased}
\overline{\mu}\left( \bz\right)=\frac{\bz\cdot\bu\,\overline{\nu}\left( \bz\right)}{\sum_{\y\in\NN\setminus\{\z\}}\y\cdot\bu\,\overline{\nu}\left( \y\right)},\ \ \bz\in\NN\setminus\{\z\},
\end{equation}
where $\overline{\nu}$ is the Yaglom distribution associated with the subcritical process $\big(\overline{\X}_k\big)_{k\geqslant 0}$.

\subsection{A Yaglom-type distribution  (\cite{Naka78}, Theorem 3)}
\label{sec:Yaglom type}

Let $\left(\X_k\right)_{k\geqslant 0} $ be a noncritical multitype GW process satisfying $(A_1)$. Then for all $\x_0,\bz\in\NN\setminus\{\z\}$ and $n\in\N$,
$
\lim_{k}\PP_{\x_0}\left( \X_{k}=\bz \mid \X_{k+n}\neq\z,\,T<+\infty\right)=\overline{\nu}^{(n)}(\bz)$,  
where $\overline{\nu}^{(n)}$ is a probability distribution on $\NN\setminus\{\z\}$ independent of the initial state $\x_0$. In particular, $\overline{\nu}^{(0)}=\overline{\nu}$ is the Yaglom distribution associated with $\big(\overline{\X}_k\big)_{k\geqslant 0}$, the associated subcritical process with respect to $\q$. Moreover, assuming in addition $(A_2)$ if $\rho<1$, then for each $\bz\in\NN\setminus\{\z\}$, $
\lim_{n}\overline{\nu}^{(n)}(\bz)=\overline{\mu}\left( \bz\right).$

\section{Conditioning on reaching a certain state or threshold}
\label{sec:threshold}
In this section we shall generalize \eqref{limext} by proving that by replacing the conditioning event  $\{\X_{k_j+n}\neq \z\}$ by $\{\X_{k_j+n}\in S\}$, where $S$ is a subset of $\NN\setminus\{\z\}$, the obtained limit process remains the $Q$-process. In particular, conditioning the process on reaching a certain non-zero state or positive threshold in a distant future, i.e. with \begin{equation*}S=\{\y\},\ S=\{\x\in\NN,\ \|\x\|_1= m\}\ \mbox{or}\ S=\{\x\in\NN,\ \|\x\|_1\geqslant m\},\end{equation*} ($\y\neq\z, m>0$), leads to the same result as conditioning the process on non-extinction.

In what follows we call a subset $S$ accessible if for any $\x\in\NN\setminus\{\z\}$, there exists some $n\in\mathbb{N}$ such that $\PP_{\x}\left( \X_n\in S\right) >0$. For any subset $S$ we shall denote $S^c=\NN\setminus\left(\{\z\}\cup S\right)$. 
\begin{theorem}Let $\left(\X_k\right)_{k\geqslant 0} $ be a multitype GW process satisfying $(A_1)$, and let $S$ be a subset of $\NN\setminus\{\z\}$. If $\rho\leqslant 1$ we assume in addition one of the following assumptions:
\begin{itemize}
\item[$(a_1)$] $S$ is finite and accessible,
\item[$(a_2)$] $S^c$ is finite,
\item[$(a_3)$] $\left(\X_k\right)_{k\geqslant 0} $ is subcritical and satisfies $(A_2)$.
\end{itemize}
Then for all $\x_0\in\NN\setminus\{\z\}$, $k_1\leqslant\ldots\leqslant k_j\in\mathbb{N}^*$ and  $\x_1,\ldots,\x_j\in\NN$,
\begin{multline}\label{result}
\lim_{n\to +\infty}\PP_{\x_0}\left( \X_{k_1}=\x_1,\ldots, \X_{k_j}=\x_j\mid\X_{k_j+n}\in S,\ T<+\infty\right)
\\=\frac{1}{\overline{\rho}^{k_j}}\frac{\x_j\cdot\overline{\bu}}{\x_0\cdot\overline{\bu}}\PP_{\x_0}\left( \overline{\X}_{k_1}=\x_1,\ldots,
 \overline{\X}_{k_j}=\x_j\right),
\end{multline}
where $\big(\overline{\X}_k\big)_{k\geqslant 0}$ is  the associated process with respect to $\q$. 
\end{theorem}

\begin{proof}Note that if $\rho>1$, then $\q<\1$ (\cite{AthNey}) which implies that $\mathbb{E}_{\mathbf{e}_i}( \overline{X}_{1,j}\ln \overline{X}_{1,j}) <+\infty$, meaning that $\big(\overline{\X}_k\big)_{k\geqslant 0} $ automatically satisfies $(A_2)$. Thanks to Remark \ref{rem: subcritic}, we can thus assume without loss of generality that $\rho \leqslant 1$ and simply consider the limit \begin{multline}\label{reaching}\lim_{n\to +\infty}\PP_{\x_0}\left( \X_{k_1}=\x_1,\ldots, \X_{k_j}=\x_j\mid\X_{k_j+n}\in S\right)\\=\lim_{n\to +\infty}\PP_{\x_0}\left( \X_{k_1}=\x_1,\ldots, \X_{k_j}=\x_j\right) \frac{\PP_{\x_j}\left( \X_n\in S\right)}{\PP_{\x_0}\big( \X_{k_j+n}\in S\big)}.\end{multline}

Let us recall here some of the technical results established for $\rho \leqslant 1$ in \cite{Naka78}, essential to our proof. First, for each $\z\leqslant \bb<\bc\leqslant \mathbf{1}$ and $\x\in\NN$,
\begin{equation}\label{Nak1}
\lim_{n\to +\infty}\frac{\bv\cdot\left(  \ef_{n+2}\left( \z\right)-\ef_{n+1}\left( \z\right) \right) }{\bv\cdot\left(  \ef_{n+1}\left( \z\right)-\ef_{n}\left( \z\right) \right) }=\rho,
\end{equation}
\begin{equation}\label{Nak2}
\lim_{n\to +\infty}\frac{  \ef_{n}\left( \bc\right)^{\x}-\ef_{n}\left( \bb\right)^{\x} }{\bv\cdot\left(  \ef_{n}\left( \bc\right)-\ef_{n}\left( \bb\right) \right) }=\x\cdot\bu,
\end{equation}
Moreover, for each $\x,\y\in\NN\setminus\{\z\}$,
\begin{equation}\label{Nak3}
\lim_{n\to +\infty}\frac{1-\ef_{n+1}\left( \z\right) ^{\x}}{1-\ef_{n}\left( \z\right) ^{\x}}=\rho,
\end{equation}
\begin{equation}\label{Nak4}
P_n\left( \x,\y\right) =\left( \pi\left( \y\right) +\varepsilon_n\left(  \x,\y\right)\right) \left( \ef_{n+1}\left( \z\right) ^{\x}-\ef_{n}\left( \z\right) ^{\x}\right),
\end{equation}where $\lim_{n}\varepsilon_n\left(  \x,\y\right)=0$ and $\pi$ is the unique measure (up to multiplicative constants) on $\NN\setminus\{\z\}$ not identically zero satisfying $\sum_{\y\neq \z}\pi( \y) P( \y,\bz)=\rho\pi( \bz)$ for each  $\bz\neq \z$. In particular, if $\rho<1$, $\pi=\left( 1-\rho\right)^{-1}\nu$, where $\nu$ is the probability distribution defined by \eqref{Yaglom}.

Let us first assume $(a_1)$. By \eqref{Nak4}
\begin{align}\label{toworkon1}
\frac{\PP_{\x_j}\left( \X_n\in S\right)}{\PP_{\x_0}\big( \X_{k_j+n}\in S\big)}&=\frac{\sum_{\bz \in S}P_n\left( \x_j,\bz\right) }{\sum_{\bz \in S}P_{n+k_j}\left( \x_0,\bz\right)}\nonumber\\&=\frac{\pi\left( S\right) +\varepsilon_n\left( \x_j\right) }{\pi\left( S\right) +\varepsilon_{n+k_j}\left( \x_0\right)}\frac{ \ef_{n+1}\left( \z\right) ^{\x_j}-\ef_{n}\left( \z\right) ^{\x_j}}{\ef_{n+k_j+1}\left( \z\right) ^{\x_0}-\ef_{n+k_j}\left( \z\right) ^{\x_0}},
\end{align}
where $\lim_{n}\varepsilon_n\left( \x\right)=\lim_{n}\sum_{\bz\in S}\varepsilon_n\left(  \x,\bz\right)=0$ since $S$ is finite. On the one hand, we can deduce from \eqref{Nak1} and \eqref{Nak2} that
\begin{equation*}
\lim_{n\to +\infty}\frac{ \ef_{n+1}\left( \z\right) ^{\x_j}-\ef_{n}\left( \z\right) ^{\x_j}}{\ef_{n+k_j+1}\left( \z\right) ^{\x_0}-\ef_{n+k_j}\left( \z\right) ^{\x_0}}
=\frac{1}{\rho^{k_j}}\frac{\x_j\cdot\bu}{\x_0\cdot\bu}.
\end{equation*}
On the other hand, $\pi$ being not identically zero, there exists some $\y_0\in\NN\setminus\{\z\}$ such that $\pi\left( \y_0\right) >0$. Since $S$ is accessible, there exists some $\bz_0\in S$ and $k\in\mathbb{N}^*$ such that $P_k\left( \y_0,\bz_0\right)>0$, and thus
\[+\infty>\pi\left( S\right) \geqslant\pi\left(  \bz_0\right) =\rho^{-k}\sum_{\y\in\NN\setminus\{\z\}}\pi\left(  \y\right)  P_k\left( \y,\bz_0\right)\geqslant \rho^{-k}\pi\left(  \y_0\right) P_k\left( \y_0,\bz_0\right)>0.\]
From \eqref{toworkon1} we thus deduce that \eqref{reaching} leads to \eqref{result}.

Let us now assume $(a_2)$. We can similarly deduce from \eqref{Nak4} that
\begin{multline}\label{toworkon}
\frac{\PP_{\x_j}\left( \X_n\in S\right)}{\PP_{\x_0}( \X_{k_j+n}\in S)}\\=\frac{1-\ef_{n}\left( \z\right) ^{\x_j}- \left( \pi\left( S^c\right) +\varepsilon_n\left( \x_j\right)\right) \left( \ef_{n+1}\left( \z\right) ^{\x_j}-\ef_{n}\left( \z\right) ^{\x_j}\right) }{1-\ef_{n+k_j}\left( \z\right) ^{\x_0}-( \pi\left( S^c\right) +\varepsilon_{n+k_j}\left( \x_0\right)) (\ef_{n+k_j+1}\left( \z\right) ^{\x_0}-\ef_{n+k_j}\left( \z\right) ^{\x_0}) },
\end{multline}
with $0\leqslant \pi\left( S^c\right) <+\infty$ and $\lim_{n}\varepsilon_n\left( \x\right)=\lim_{n}\sum_{\bz\in S^c}\varepsilon_n\left(  \x,\bz\right)=0$ since $S^c$ is finite. Note that \eqref{limext} implies that 
\begin{equation*}
\ \lim_{n\to +\infty}\frac{1-\ef_{n}\left( \z\right) ^{\x_j}}{1-\ef_{n+k_j}\left( \z\right) ^{\x_0}}=\frac{1}{\rho^{k_j}}\frac{\x_j\cdot\bu}{\x_0\cdot\bu},
\end{equation*} which together with \eqref{Nak3} enables to show that \eqref{toworkon} tends to $\rho^{-k_j}\frac{\x_j\cdot\bu}{\x_0\cdot\bu}$ as $n$ tends to infinity, leading again to \eqref{result}.

Let us finally assume $(a_3)$. Then we know from \cite{Naka78} (Remark 2) that $\pi(\bz)>0$ for each $\bz\neq \z$, hence automatically $0<\pi\left( S\right) =\left( 1-\rho\right)^{-1}\nu\left( S\right)<+\infty$. Moreover, $\nu$ admits finite first-order moments (see \eqref{first moment g}). Hence for any $a>0$, by Markov's inequality,
\begin{align*}
\left| \sum_{\bz\in S}\varepsilon_n\left( \x,\bz\right)\right|&\leqslant \sum_{\substack{\bz\in S\\\|\bz\|_1<a}}\left|\varepsilon_n\left(  \x,\bz\right)\right|+\sum_{\substack{\bz\in S\\\|\bz\|_1 \geqslant a}}\left|\frac{P_n\left( \x,\bz\right)}{\ef_{n+1}\left( \z\right) ^{\x}-\ef_{n}\left( \z\right) ^{\x}}- \pi\left( \bz\right)\right|\\&\leqslant \sum_{\substack{\bz\in S\\\|\bz\|_1<a}}\left|\varepsilon_n\left(  \x,\bz\right)\right|+\frac{1}{a}\frac{\EE_{\x}\left(\|\X_n\|_1\right)}{\ef_{n+1}\left( \z\right) ^{\x}-\ef_{n}\left( \z\right) ^{\x}}+\frac{1}{1-\rho}\frac{1}{a}
\sum_{i=1}^d\frac{\partial g \left(  \mathbf{1}\right) }{\partial r_i}.
\end{align*}
We recall that by \eqref{againbasic}, $\lim_n \rho^{-n}\left( \ef_{n+1}\left( \z\right) ^{\x}-\ef_{n}\left( \z\right) ^{\x}\right) =\left( 1-\rho\right) \gamma\,\x\cdot\bu$, while  by \eqref{mean}, $\lim_n \rho^{-n}\EE_{\x}\left(\|\X_n\|_1\right)=\sum_{i,j=1}^dx_iu_iv_j$. Hence the previous inequality ensures that $\lim_n \sum_{\bz\in S}\varepsilon_n\left( \x,\bz\right)=0$. We can thus write \eqref{toworkon1} even without the finiteness assumption of $S$, and prove \eqref{result} as previously. 
\end{proof}

\section{The size-biased Yaglom distribution as a double limit}
\label{sec:Yaglom double}

From Subsection \ref{sec:Q process} and Subsection \ref{sec:Yaglom type} we know that in the noncritical case, assuming $(A_2)$ if $\rho<1$,
\begin{align*}
\lim_{k\to +\infty}\lim_{n\to +\infty}\PP_{\x_0}\left(  \X_k=\bz\mid\X_{k+n}\neq \z,\,T<+\infty \right) &=\lim_{k\to +\infty}Q_k\left(\x_0, \bz\right) =\overline{\mu}\left( \bz\right),\\
   \lim_{n\to +\infty}\lim_{k\to +\infty}\PP_{\x_0}\left(  \X_k=\bz\mid\X_{k+n}\neq \z,\,T<+\infty \right) &=\lim_{n\to +\infty}\overline{\nu}^{(n)}\left(\bz\right) =\overline{\mu}\left( \bz\right).
\end{align*} We prove here that, under the stronger assumption $(A_3)$ if $\rho<1$, this limiting result also holds when $k$ and $n$ simultaneously tend to infinity. 

\begin{theorem}Let $\left(\X_k\right)_{k\geqslant 0} $ be a noncritical multitype GW process satisfying $(A_1)$. If $\rho<1$, we assume in addition $(A_3)$. Then  for all $\x_0\in\NN\setminus\{\z\}$ and $\bz\in\NN$, \[\lim_{\substack{n\to +\infty\\k\to+\infty}}\PP_{\x_0}\left(  \X_k=\bz\mid\X_{k+n}\neq \z,\,T<+\infty \right) =\overline{\mu}\left( \bz\right),\] where $\overline{\mu}$ is the size-biased Yaglom distribution of  $\big(\overline{\X}_k\big)_{k\geqslant 0}$,  the associated process with respect to $\q$. \end{theorem}
\begin{remark}
This implies in particular that for any $0<t<1$,
\[\lim_{k\to +\infty}\PP_{\x_0}\left(  \X_{\lf kt\rf}=\bz\mid\X_{k}\neq \z,\,T<+\infty \right) =\overline{\mu}\left( \bz\right).\]
\end{remark}

\begin{remark}In the critical case, the $Q$-process is transient and the obtained limit is degenerate. A suitable normalization in order to obtain a non-degenerate probability distribution is of the form $\X_k/k$. However, even with this normalization, the previous result does not hold in the critical case. Indeed, we know for instance that in the monotype case, a critical process with finite variance $\sigma^2>0$ satisfies for each $z\geqslant 0$ (\cite{LaNey68}),
\begin{align*}\lim_{k\to +\infty}\lim_{n\to +\infty}\PP_{1}\left(  \frac{X_k}{k}\leqslant z\mid X_{k+n}\neq 0\right) &=1-e^{-\frac{2z}{\sigma^2}} ,\\
\lim_{n\to +\infty}\lim_{k\to +\infty}\PP_{1}\left(  \frac{X_k}{k}\leqslant z\mid X_{k+n}\neq 0 \right)&=1-e^{-\frac{2z}{\sigma^2}}-\frac{2z}{\sigma^2}e^{-\frac{2z}{\sigma^2}}.
\end{align*}
\end{remark}

\begin{proof}Thanks to Remark \ref{rem: subcritic} and to the fact that if $\rho>1$, $\mathbb{E}_{\mathbf{e}_i}( \overline{X}_{1,j} \overline{X}_{1,l})<+\infty$, we can assume without loss of generality that $\rho < 1$. For each $n$, $k\in\mathbb{N}$ and $\er\in[0,1]^d$, \[\EE_{\x_0}\left(\er^{\X_k}\mathbf{1}_{\X_{k+n}=\z} \right)=\sum_{\y\in\NN}\PP_{\x_0}\left(\X_k=\y\right) \er^{\y}\PP_{\y}\left(\X_{n}=\z \right)=\ef_k\left( \er\ef_{n}\left( \z\right) \right)^{\x_0},\] which leads to
\begin{align}\label{begin}
\EE_{\x_0}\left[\er^{\X_k}\mid \X_{k+n}\neq \z  \right]&=\frac{\EE_{\x_0}\left(\er^{\X_k} \right) -\EE_{\x_0}\left(\er^{\X_k}\mathbf{1}_{\X_{k+n}=\z} \right)}{1-\PP_{\x_0}\left(\X_{k+n}= \z  \right)}\nonumber\\&=\frac{\ef_k\left( \er\right)^{\x_0} -\ef_k\left( \er\ef_{n}\left( \z\right) \right)^{\x_0} }{1-\ef_{k+n}\left( \z\right)^{\x_0}}.
\end{align}
By Taylor's theorem,
\begin{multline}\label{Taylor}\ef_k\left( \er\right)^{\x_0} -\ef_k\left( \er\ef_{n}\left( \z\right) \right)^{\x_0}=\sum_{i=1}^d\frac{\partial \ef_k^{\x_0}\left( \er\right)}{\partial r_i}r_i\left( 1-f_{n,i}\left( \z\right) \right) \\-\sum_{i,j=1\ldots d} \frac{r_ir_j\left( 1-f_{n,i}\left( \z\right) \right)( 1-f_{n,j}\left( \z\right) )}{2}\int_0^1\left( 1-t\right) \frac{\partial ^2\ef_k^{\x_0}\left( \er-t \er\left( \mathbf{1}-\ef_{n}\left( \z\right)\right) \right)}{\partial r_i\partial r_j}dt,
\end{multline}
with 
\begin{equation}\label{calculus}\frac{\partial \ef_k^{\x_0}\left( \er\right)}{\partial r_i}=\sum_{j=1}^dx_{0,j}\frac{\partial f_{k,j}\left( \er\right) }{\partial r_i}\ef_k\left( \er\right)^{\x_0-\mathbf{e}_j}.
\end{equation}

 Let us first prove the existence of $\lim_k\rho^{-k}\frac{\partial  f_{k,j}\left(\er\right)}{\partial r_i}$ for each $i,j$ and $\er\in[0,1]^d$. For each $k$, $p\in\mathbb{N}$ and $a>0$,
\begin{multline}\label{en fait si}
 \Big| \rho^{-k}\frac{\partial f_{k,j}\left(\er\right)}{\partial r_i}- \rho^{-(k+p)}\frac{\partial f_{k+p,j}\left(\er\right)}{\partial r_i} \Big|\\\leqslant \sum_{\substack{\bz\in\NN\\\|\bz\|_2< a}}z_i\er^{\bz-\mathbf{e}_i}\Big|\rho^{-k}P_k(\mathbf{e}_j, \bz) - \rho^{-(k+p)}P_{k+p}(\mathbf{e}_j, \bz)\Big| \\+ \rho^{-k}\EE_{\mathbf{e}_j}\left( X_{k,i}\mathbf{1}_{\|\X_k\|_2\geqslant a}\right) +\rho^{-(k+p)}\EE_{\mathbf{e}_j}\left( X_{k+p,i}\mathbf{1}_{\|\X_{k+p}\|_2\geqslant a}\right).
\end{multline}
By Cauchy-Schwarz and Markov's inequalities, $ \EE_{\mathbf{e}_j}( X_{k,i}\mathbf{1}_{\|\X_k\_2|\geqslant a})\leqslant \frac{1}{a}\EE_{\mathbf{e}_j}( \|\X_{k}\|_2^2).$ For each $\x\in\NN$, let  $\mathbf{C}_{\x,k}$ be the matrix $( \mathbb{E}_{\x}( X_{k,i}X_{k,j}))_{1\leqslant i,j\leqslant d}$. According to \cite{Har63}, 
\begin{equation}\label{har}
\mathbf{C}_{\x,k}=( \mathbf{M}^T) ^k\mathbf{C}_{\x,0}\mathbf{M}^k+\sum_{n=1}^k( \mathbf{M}^T) ^{k-n}\left( \sum_{i=1}^d\mathbf{\Sigma}^{i}\EE_{\x}\left(X_{n-1,i} \right)  \right) \mathbf{M}^{k-n}.
\end{equation}
Thanks to \eqref{mean} this implies the existence of some $C>0$ such that for all $k\in\mathbb{N}$, 
$\rho^{-k}\EE_{\mathbf{e}_j}( \|\X_{k}\|_2^2)=\rho^{-k}\sum_{i=1}^d[ \mathbf{C}_{\mathbf{e}_j,k}]_{ii} \leqslant C$, and the two last right terms in \eqref{en fait si} can be bounded by $2Ca^{-1}$. As for the first right term in  \eqref{en fait si}, it is thanks to \eqref{againbasic} and \eqref{Nak4} as small as desired for $k$ large enough. This proves that $(  \rho^{-k}\frac{\partial f_{k,j}\left(\er\right)}{\partial r_i})_k$ is a Cauchy sequence.

Its limit is then necessarily, for each  $\er\in[0,1]^d$,
\begin{equation}\label{limit to prove}
\lim_{k\to+\infty} \rho^{-k}\frac{\partial f_{k,j}\left(\er\right)}{\partial r_i}=\gamma u_j\frac{\partial g\left(\er\right)}{\partial r_i},
\end{equation}
where $g$ is defined in Subsection \ref{sec:nonext}. Indeed, since assumption $(A_3)$ ensures that $(A_2)$ is satisfied, we can deduce from \eqref{Yaglom} and \eqref{againbasic}
 that $\lim_{k}\rho^{-k}(f_{k,j}( \er)-f_{k,j}( \z))=\gamma u_j g(\er)$. Hence, using the fact that $0\leqslant \rho^{-k}\frac{\partial f_{k,j}\left( \er\right)}{\partial r_i} \leqslant  \rho^{-k}m^{(k)}_{ji}$,  which thanks to  \eqref{mean} is bounded, we obtain by Lebesgue's dominated convergence theorem that for each $h\in\mathbb{R}$ such that $\er+h\mathbf{e}_i\in[0,1]^d$, \begin{equation*}\gamma u_j g(\er+h\mathbf{e}_i)-\gamma u_j g(\er)
=\lim_{k\to+\infty} \int_0^h\rho^{-k}\frac{\partial  f_{k,j}\left(\er+t\mathbf{e}_{i}\right)}{\partial r_i}dt= \int_0^h\lim_{k\to+\infty}\rho^{-k}\frac{\partial  f_{k,j}\left(\er+t\mathbf{e}_{i}\right)}{\partial r_i}dt,\end{equation*}
proving \eqref{limit to prove}.

In view of \eqref{Taylor}, let us note that for each $\er\in[0,1]^ d$, there exists thanks to \eqref{mean} and \eqref{har} some $C>0$ such that for each $k\in\mathbb{N}$,
$ 0\leqslant \rho^{-k}\frac{\partial ^2\ef_k^{\x}\left( \er \right)}{\partial r_i\partial r_j}\leqslant\rho^{-k}\EE_{\x}[ X_{k,j}( X_{k,i}-\delta_{ij})]\leqslant C$, hence  for each $k$, $n\in\mathbb{N}$,\[\rho^{-k}\int_0^1\left( 1-t\right) \frac{\partial ^2\ef_k^{\x}\left( \er-t \er\left( \mathbf{1}-\ef_{n}\left( \z\right)\right) \right)}{\partial r_i\partial r_j}dt\leqslant \frac{C}{2}.\]
Together with \eqref{againbasic} this entails that the last right term in  \eqref{Taylor} satisfies
\[\lim_{\substack{n\to +\infty\\k\to+\infty}}\rho^{-(k+n)}\sum_{i,j=1\ldots d}\frac{r_ir_j\left( 1-f_{n,i}\left( \z\right) \right)( 1-f_{n,j}\left( \z\right) )}{2} \int_0^1\ldots\ dt=0.\]
Moreover, we deduce from \eqref{calculus}, \eqref{limit to prove} and $\lim_n\ef_n(\er)=\mathbf{1}$ that the first right term in  \eqref{Taylor} satisfies
\[\lim_{\substack{n\to +\infty\\k\to+\infty}}\rho^{-(k+n)}\sum_{i=1}^d\frac{\partial \ef_k^{\x_0}\left( \er\right)}{\partial r_i}r_i\left( 1-f_{n,i}\left( \z\right) \right)=\gamma^2\,\x_0\cdot\bu \sum_{i=1}^dr_iu_i\frac{\partial g\left(\er\right)}{\partial r_i}.\]
Recalling \eqref{begin} and \eqref{againbasic}, we have thus proven that for each $\er\in[0,1]^d$,
\[\lim_{\substack{n\to +\infty\\k\to+\infty}}\EE_{\x_0}\left[\er^{\X_k}\mid \X_{k+n}\neq \z  \right] =\gamma \sum_{i=1}^dr_iu_i\frac{\partial g\left(\er\right)}{\partial r_i}=\gamma \sum_{\bz\in\NN\setminus\{\z\}}\bz\cdot\bu\,\nu\left( \bz\right)\er^{\bz}.\]
Finally, \eqref{first moment g} leads to
$\gamma \sum_{i=1}^du_i\frac{\partial g\left(\mathbf{1}\right)}{\partial r_i}=1$, and thus
\[\lim_{\substack{n\to +\infty\\k\to+\infty}}\EE_{\x_0}\left[\er^{\X_k}\mid \X_{k+n}\neq \z  \right] =\frac{\sum_{\bz\in\NN\setminus\{\z\}}\bz\cdot\bu\,\nu\left( \bz\right)\er^{\bz}}{\sum_{\y\NN\setminus\{\z\}}\y\cdot\bu\, \nu\left( \y\right)},\]
which by \eqref{size biased} is a probability generating function.
\end{proof}

\section{Conditioning on the total progeny}
\label{sec:totalprog}

Let $\mathbf{N}=\left( N_1,\ldots,N_d\right)$ denote the total progeny of the process $\left(\X_k\right)_{k\geqslant 0} $, where for each $i=1\ldots d$,
$N_i=\sum_{k=0}^{+\infty}X_{k,i}$,
and $N_i=+\infty$ if the sum diverges. Our aim is to study the behavior of $\left(\X_k \right)_{k\geqslant 0} $ conditioned on the event $\{\N=\lf n\mathbf{w} \rf\}$, as $n$ tends to infinity, for some specific positive vector $\mathbf{w}$. We recall that in the critical case, the GW process suitably normalized and conditioned on non-extinction in the same fashion as in \eqref{Yaglom}, converges to a limit law supported by the ray $\{\lambda\bv: \lambda\geqslant 0\}\subset \mathbb{R}_+^d$. In this sense,
its left eigenvector $\bv$ describes "typical limiting type proportions", as pointed out in \cite{FleiVat06}. As we will see in Lemma \ref{lem1}, conditioning a GW process on a given total progeny size comes down to conditioning an associated critical process on the same total progeny size. For this reason, the vector $\mathbf{w}$ will be chosen to be the left eigenvector of the associated critical process. It then appears that, similarly as in the monotype case (\cite{Ken75}), the process conditioned on an infinite total progeny $\{\N=\lf n\mathbf{w} \rf\}$, $n\to\infty$, has the structure of the $Q$-process of a critical process, and is consequently transient. This is the main result, stated in Theorem \ref{thm2}.

\begin{theorem}\label{thm2} Let $\left(\X_k\right)_{k\geqslant 0} $ be a multitype GW process satisfying $(A_1)$. We assume in addition that
\begin{enumerate}
\item[$(A_4)$] there exists $\ba>\z$ such that the associated process with respect to $\ba$ is critical,
\item[$(A_5)$] for each $j=1\ldots d$, there exist $i=1\ldots d$ and $\ka\in\NN$ such that $p_i\left(\ka\right) >0$ and $p_i\left(\ka+\mathbf{e}_j\right) >0$,
\item[$(A_6)$] the associated process with respect to $\ba$ admits moments of order $d+1$, and its covariance matrices are positive-definite.
\end{enumerate}
Then for all $\x_0\in\NN\setminus\{\z\}$, $k_1\leqslant\ldots\leqslant k_j\in\mathbb{N}$,  and $\x_1,\ldots,\x_j\in\NN$,
\begin{equation}\label{lim}
\lim_{n\to+\infty}\PP_{\x_0}\big( \X_{k_1}=\x_1,\ldots, \X_{k_j}=\x_j\mid\N=\lf n\overline{\bv} \rf\big)\\=\frac{\x_j\cdot\overline{\bu}}{\x_0\cdot\overline{\bu}}\PP_{\x_0}\big( \overline{\X}_{k_1}=\x_1,\ldots, \overline{\X}_{k_j}=\x_j\big),
\end{equation}
where $\big(\overline{\X}_k\big)_{k\geqslant 0}$ is  the associated process with respect to $\ba$.
\end{theorem}
 
The limiting process defined by \eqref{lim} is thus Markovian with transition probabilities
\begin{equation*}
\overline{Q}_1\left( \x,\y\right) =\frac{\y\cdot\overline{\bu}}{\x\cdot\overline{\bu}}\overline{P}_1\left( \x,\y\right)=\frac{\ba^{\y}}{\ef\left( \ba\right) ^{\x}}\frac{\y\cdot\overline{\bu}}{\x\cdot\overline{\bu}}P_1\left( \x,\y\right),\ \ \ \ \ \ \x,\y\in\NN\setminus\{\z\},
\end{equation*}and corresponds to the $Q$-process associated with the critical process $\big(\overline{\X}_k\big)_{k\geqslant 0}$.

\begin{remark}$ $
\begin{itemize}
\item If $d=1$, the conditional event $\{\N=\lf n\overline{\bv} \rf\}$ reduces to $\{N=n\}$, as studied in \cite{Ken75}, in which assumptions $(A_4)$--$(A_6)$ are also required\footnote{Since the author's work, it has been proved in \cite{AbrDelGuo15} that in the critical case and under $(A_1)$, Theorem \ref{thm2} holds true under the minimal assumptions of aperiodicity of the offspring distribution (implied by $(A_5)$) and the finiteness of its first order moment.}.  
\item   If $\left(\X_k\right)_{k\geqslant 0} $ is critical, assumption $(A_4)$ is satisfied with $\ba=\1$. This assumption is also automatically satisfied if $\left(\X_k\right)_{k\geqslant 0} $ is supercritical. Indeed, as mentioned in Remark \ref{rem: subcritic}, the associated process with respect to $\z<\q<\1$ is subcritical and thus satisfies $\overline{\rho}< 1$. The fact that $\rho>1$ and the continuity of the Perron's root as a function of the mean matrix coefficients then ensures the existence of some $\q\leqslant \ba \leqslant \1$ satisfying $(A_4)$. Note however that such an $\ba$ is not unique.
\item For any $\ba>\z$, $p_i$ and $\overline{p}_i$ share by construction the same support. As a consequence, $\left(\X_k\right)_{k\geqslant 0} $ satisfies $(A_5)$ if and only if $\left(\overline{\X}_k\right)_{k\geqslant 0} $ does. Moreover, a finite covariance matrix $\mathbf{\Sigma}^i$ is positive-definite if and only if there does not exist any  $c\in\mathbb{R}$ and $\x\neq \z$ such that $\x\cdot\X=c$ $\PP_{\mathbf{e}_i}$-almost-surely, hence if and only if  $\x\cdot\overline{\X}=c$ $\PP_{\mathbf{e}_i}$-almost-surely. Consequently, provided it exists, $\mathbf{\Sigma}^i$ is positive-definite if and only if $\overline{\mathbf{\Sigma}}^i$ is positive-definite as well.
\end{itemize}
\end{remark}

We shall first show in Lemma \ref{lem1} that for any $\ba$, the associated process $\big(\overline{\X}_k\big)_{k\geqslant 0}$ with respect to $\ba$, conditioned on $\{\overline{\N}=\n\}$, has the same probability distribution as the original process conditioned on $\{\N=\n\}$, for any $\n\in\NN$. It is thus enough to prove Theorem \ref{thm2} in the critical case, which is done at the end of the article.

 It follows from Proposition 1 in \cite{Good75} or directly from Theorem 1.2 in \cite{ChauLiu11} that the  probability distribution of the total progeny in the multitype case is given for each $\x_0$, $\n\in\NN$ with $\n>\z$, $\n\geqslant \x_0 $ by
\begin{equation}\label{totprogeny}
\PP_{\x_0 }\left( \N=\n\right)=\frac{1}{n_1\ldots n_d}\sum_{\substack{\ka^{1},\ldots,\ka^{d}\in\NN\\\ka^{1}+\ldots+\ka^{d}=\n-\x_0}}\det \begin{pmatrix}
n_1\mathbf{e}_1-\ka^1\\\cdots \\n_{d}\mathbf{e}_{d}-\ka^{d}
\end{pmatrix}   \prod_{i=1}^d p_i^{*n_i}\left( \ka^i\right).\end{equation}

\begin{lemma}\label{lem1}Let $\left(\X_k \right)_{k\geqslant 0} $ be a multitype GW process. Then, for any $\ba>\z$, the associated process $\big(\overline{\X}_k\big) _{k\geqslant 0}$ with respect to $\ba$ satisfies for any $\x_0\in\NN$, $k_1\leqslant\ldots\leqslant k_j\in\mathbb{N}$, $\x_1,\ldots,\x_j\in\NN$ and $\n\in\NN$,
\begin{equation*}
\PP_{\x_0}\left( \X_{k_1}=\x_1,\ldots, \X_{k_j}=\x_j\mid\N=\n\right)=\PP_{\x_0}\left( \overline{\X}_{k_1}=\x_1,\ldots, \overline{\X}_{k_j}=\x_j\mid\overline{\N}=\n\right).
\end{equation*}
\end{lemma}

\begin{proof}From \eqref{off} and \eqref{totprogeny} , 
$\PP_{\ka}\left( \overline{\N}=\n\right)
 =\frac{\ba^{\n-\ka}}{\ef\left( \ba\right)^{\n} } \PP_{\ka}\left( \N=\n\right)$.
  For all $n\in\mathbb{N}$, we denote by $\N_n=\sum_{k=0}^n\X_k$ (resp. $\overline{\N}_n=\sum_{k=0}^n\overline{\X}_k$) the total progeny up to generation $n$ of $\left(\X_k \right)_{k\geqslant 0} $  (resp. $\big(\overline{\X}_k\big)_{k\geqslant 0}$). Then
\begin{align*}
 \PP_{\x_0}\left(\overline{\X}_{k_j}=\x_j,\overline{\N}_{k_j}=\el\right)&=\sum_{\substack{\bi_1,\ldots,\bi_{k_j-1}\in\NN\\\bi_1+\ldots+\bi_{k_j-1}=\el-\x_0-\x_j}}\overline{P}_1\left( \x_0,\bi_1\right) \ldots\overline{P}_1( \bi_{k_j-1},\x_j) \\
&=\sum_{\substack{\bi_1,\ldots,\bi_{k_j-1}\in\NN\\\bi_1+\ldots+\bi_{k_j-1}=\el-\x_0-\x_j}}\frac{\ba^{\bi_1}P_1\left( \x_0,\bi_1\right) }{\ef\left( \ba\right) ^{\x_0}}\ldots\frac{\ba^{\x_j}P_1( \bi_{k_j-1},\x_j)  }{\ef\left( \ba\right) ^{\bi_{k_j-1}}}\\&=\frac{\ba^{\el-\x_0}\PP_{\x_0}\left(\X_{k_j}=\x_j,\N_{k_j}=\el\right)}{\ef\left( \ba\right) ^{\el-\x_j}} ,
\end{align*}
and similarly
\begin{equation*}
 \PP_{\x_0}\Big( \overline{\X}_{k_1}=\x_1,\ldots, \overline{\X}_{k_j}=\x_j,\overline{\N}_{k_j}=\el\Big)\\=\frac{\ba^{\el-\x_0}\PP_{\x_0}\left( \X_{k_1}=\x_1,\ldots, \X_{k_j}=\x_j,\N_{k_j}=\el\right)}{\ef\left( \ba\right) ^{\el-\x_j}}  .
\end{equation*}
Consequently, thanks to the Markov property,
\begin{align*}
&\PP_{\x_0}\left( \overline{\X}_{k_1}=\x_1,\ldots, \overline{\X}_{k_j}=\x_j\mid\overline{\N}=\n\right)
\\
&\ \ =\sum_{\substack{\el\in\NN\\\el\leqslant \n}}\frac{\PP_{\x_0}\left( \overline{\X}_{k_1}=\x_1,\ldots, \overline{\X}_{k_j}=\x_j,\overline{\N}_{k_j}=\el\right)\PP_{\x_j}\left(\overline{\N}=\n-\el+\x_j\right)}{\PP_{\x_0}\left(\overline{\N}=\n\right)} \\
&\ \ =\PP_{\x_0}\left( \X_{k_1}=\x_1,\ldots, \X_{k_j}=\x_j\mid\N=\n\right).
\end{align*}
\end{proof}

Thanks to Lemma \ref{lem1}, it suffices to prove Theorem \ref{thm2} in the critical case. For this purpose, we prove the following convergence result for the total progeny of a critical GW process.
\begin{proposition}\label{prop: convergence}
 Let $\left(\X_k\right)_{k\geqslant 0} $ be a critical multitype GW process satisfying $(A_1)$, $(A_5)$ and $(A_6)$. Then there exists $C> 0$ such that for all $\x_0\in\NN$, 
\begin{equation}
\lim_{n\to+\infty}n^{\frac{d}{2}+1}\PP_{\x_0 }\left( \N=\lf n\bv \rf\right)=C\x_0\cdot \bu.
\end{equation}
\end{proposition}

\begin{proof}
From \eqref{totprogeny}, for each $n\geqslant \max_{i}v_i^{-1}$, $n\geqslant \max_{i} x_{0,i}v_i^{-1}$, 
 \begin{align*}
 \PP_{\x_0 }\left( \N=\lf n\bv \rf\right)&=\frac{1}{\prod_{i=1}^{d} \lf nv_i\rf }\EE\left[ \det \hspace{-1mm}\begin{pmatrix}
\lf nv_1\rf \mathbf{e}_1-\ES^1_{\lf nv_1\rf }\\\cdots\\\lf nv_d\rf \mathbf{e}_{d}-\ES^{d}_{\lf nv_d\rf }  
\end{pmatrix}\hspace{-1mm}\1_{\sum_{i=1}^d\ES^i_{\lf nv_i\rf }=\lf n\bv\rf -\x_0}\right]\nonumber\\
&=\frac{1}{ \lf nv_d\rf }\EE\left[\det \hspace{-1mm}\begin{pmatrix}\mathbf{e}_1-\ES^1_{\lf nv_1\rf }/\lf nv_1\rf\\\cdots\\ \mathbf{e}_{d-1}-\ES^{d-1}_{\lf nv_{d-1}\rf}/\lf nv_{d-1}\rf \\\x_0 \end{pmatrix}\hspace{-1mm}\1_{\sum_{i=1}^d\ES^i_{\lf nv_i\rf }=\lf n\bv\rf -\x_0}\right]\hspace{-1mm},
\end{align*}
where the family $(\mathbf{S}_{\lf nv_i\rf}^{i})_{i=1\ldots d}$ is independent and is such that for each $i$, $\mathbf{S}_{\lf nv_i\rf}^{i}$ denotes the sum of $\lf nv_i\rf$ independent and identically distributed random variables with  probability distribution $p_i$. 

Let us consider the event $A_n=\big\{ \sum_{i=1}^d\ES^i_{\lfloor nv_i\rfloor }=\lf n\bv\rf -\x_0\big\}$. We define the covariance matrix $\mathbf{\Sigma}=\sum_{i=1}^dv_i\mathbf{\Sigma}^{i}$, which since  $\bv>\z$ is positive-definite under $(A_6)$. 
Theorem 1.1 in \cite{Bent05} for nonidentically distributed independent variables ensures that $\sum_{i=1}^d ( \ES^i_{\lf nv_i\rf }-\lf nv_i\rf \m^i)n^{-\frac{1}{2}}$ converges in distribution as $n\to+\infty$ to the multivariate normal distribution $\mathcal{N}_d\left(\z,\mathbf{\Sigma} \right)$ with density $\phi$. Under $(A_5)$ we have
\[\limsup_n \frac{n}{\min_{j=1\ldots d}\sum_{i=1}^d \frac{n_i}{d}\sum_{\ka\in\NN}\min\left( p_i\left( \ka\right) ,  p_i\left( \ka+\mathbf{e}_j\right)  \right) }<+\infty,\]
which by Theorem 2.1 in \cite{DavMcDon} ensures the following local limit theorem for nonidentically distributed independent variables: 
\begin{equation}\label{local}
\lim_{n\to\infty}\sup_{\ka\in\NN}\left|n^{\frac{d}{2}}\PP\left(\sum_{i=1}^d \ES^i_{\lf nv_i\rf }=\ka\right)-\phi\left(\frac{\ka-\sum_{i=1}^d \lf nv_i\rf \m^i}{\sqrt{n}} \right)  \right|=0.
\end{equation}
In the critical case, the left eigenvector $\bv$ satisfies for each $j$,  $v_j=\sum_{i=1}^d v_i m_{ij}$, hence  $ 0\leqslant |\lfloor n v_j\rfloor -\sum_{i=1}^d\lfloor n v_i\rfloor m_{ij}|<\max(1,\sum_{i=1}^d m_{ij})$ and \eqref{local} implies in particular that
\begin{equation}\label{local2}
\lim_{n\to+\infty} n^{\frac{d}{2}}\PP\left(A_n\right)=\phi\left( \z\right) =\frac{1 }{\left( 2\pi\right)  ^{\frac{d}{2}}\left(\det \mathbf{\Sigma}\right) ^{\frac{1}{2}}}. 
\end{equation}

Now,  denoting by  $\mathfrak{S}_d$ the symmetric group of order $d$ and by $\epsilon(\sigma)$ the signature of a permutation $\sigma\in\mathfrak{S}_d$, we obtain by Leibniz formula that
 \begin{multline}\label{Leib}  \lf nv_d\rf\PP_{\x_0 }\left( \N=\lf n\bv \rf\right)= \sum_{\sigma\in \mathfrak{S}_d}\varepsilon\left( \sigma\right)x_{0,\sigma(d)} \EE\Big[\prod_{i=1}^{d-1}\Big( \delta_{i,\sigma(i)}-\frac{S^i_{\lf nv_i\rf,\sigma(i) }}{\lf nv_i\rf}\Big)\Big]\\=\hspace{-4mm}\sum_{I\subset\{1,\ldots, d-1\}}\sum_{\sigma\in \mathfrak{S}_d}\hspace{-2mm}\varepsilon\left( \sigma\right)x_{0,\sigma(d)} \EE\Big[\prod_{i\in I}\Big(\hspace{-1mm}-\frac{S^i_{\lf nv_i\rf,\sigma(i) }}{\lf nv_i\rf}+ m_{i,\sigma(i)}\Big)\1_{A_n}\Big]\hspace{-1mm}\prod_{i\notin I}( \delta_{i,\sigma(i)}-m_{i,\sigma(i) }).
\end{multline}
Let $\varepsilon>0$. Since on the event $A_n$ each $S_{\lf nv_i\rf,j}^{i}/\lf nv_i\rf$ is bounded, there exists some constant $A>0$ such that for each $i,j=1\ldots d$,
 \begin{align*} \EE\left(\left|\frac{S^i_{\lf nv_i\rf,j }}{\lf nv_i\rf}- m_{i,j}\right|\1_{A_n}\right)&\leqslant \varepsilon\PP\left( A_n\right) +\frac{A}{\varepsilon^{d+1}}\EE\left(\left|\frac{S^i_{\lf nv_i\rf,j}}{\lf nv_i\rf}- m_{i,j}\right|^{d+1}\right)\\&\leqslant \varepsilon\PP\left( A_n\right) +\frac{AB}{\varepsilon^{d+1}\lf nv_i\rf^{\frac{d+1}{2}}}\EE\left(\left|S^i_{1,j}- m_{i,j}\right|^{d+1}\right),
\end{align*}
for some constant $B>0$. The second inequality on the $d+1$-th central moment can be found for instance in \cite{DharJog69}, Theorem 2. From \eqref{local2} it thus appears that for each non-empty subset $I\subset\{1,\ldots, d-1\}$,
\[\lim_{n\to+\infty}n^{\frac{d}{2}}\hspace{-2mm}\sum_{\sigma\in \mathfrak{S}_d}\hspace{-2mm}\varepsilon\left( \sigma\right)x_{0,\sigma(d)} \EE\Big[\prod_{i\in I}\Big(-\frac{S^i_{\lf nv_i\rf,\sigma(i) }}{\lf nv_i\rf}+ m_{i,\sigma(i)}\Big)\1_{A_n}\Big]\hspace{-1mm}\prod_{i\notin I} \left( \delta_{i,\sigma(i)}-m_{i,\sigma(i) }\right)\hspace{-1mm}=0.\] Consequently, considering the remaining term in \eqref{Leib} corresponding to $I=\emptyset$, we obtain that
 \begin{align*}&\lim_{n\to+\infty}n^{\frac{d}{2}+1}\PP_{\x_0 }\left( \N=\lf n\bv \rf\right)\\&\ \ \ \ =\lim_{n\to+\infty}n^{\frac{d}{2}} \PP\left(A_n\right)\frac{1}{v_d}\sum_{\sigma\in \mathfrak{S}_d}\varepsilon\left( \sigma\right)x_{0,\sigma(d)}\prod_{i=1}^{d-1} \left( \delta_{i,\sigma(i)}-m_{i,\sigma(i)}\right) \\&\ \ \ \ = \frac{1 }{v_d\left( 2\pi\right)  ^{\frac{d}{2}}\left(\det \mathbf{\Sigma}\right) ^{\frac{1}{2}}}\det\begin{pmatrix}\mathbf{e}_1-\m^1\\\cdots\\ \mathbf{e}_{d-1}-\m^{d-1} \\\x_0 \end{pmatrix}=\frac{\x_0\cdot \mathbf{D}}{v_d\left( 2\pi\right)  ^{\frac{d}{2}}\left(\det \mathbf{\Sigma}\right) ^{\frac{1}{2}}},
 \end{align*}
where $\mathbf{D}=(D_1,\ldots,D_d)$ is such that $D_i$ is the $(d,i)$-th cofactor of the matrix $\mathbf{I}-\M$.
The criticality of $\left(\X_k\right)_{k\geqslant 0} $ implies that $\det\left( \mathbf{I}-\M\right)=( \mathbf{e}_d-\m^d )\cdot \mathbf{D}=0$. Moreover, for each $j=1\ldots d-1$, 
$(  \mathbf{e}_j-\m^j )\cdot \mathbf{D}$ corresponds to the determinant of $ \mathbf{I}-\M$ in which the $d$-th row has been replaced by the $j$-th row, and is consequently null. We have thus proven that for each $j=1\ldots d$, $(  \mathbf{e}_j-\m^j )\cdot \mathbf{D}=0$, or equivalently that $\sum_{i=1}^d m_{ji}  D_i=D_j.$ Hence $\mathbf{D}$ is a right eigenvector of $\M$ for the Perron's root 1, which implies the existence of some nonnull constant $c$ such that $\mathbf{D}=c\bu$, leading to the desired result.
\end{proof}

\textit{Proof of Theorem \ref{thm2}}
 Let us assume that $\left(\X_k\right)_{k\geqslant 0}$ is critical and satisfies $(A_1)$, $(A_5)$ and $(A_6)$. Let $\x_0\in\NN$, $k_1\leqslant\ldots\leqslant k_j\in\mathbb{N}$,  and $\x_1,\ldots,\x_j\in\NN$ and let us show that
\begin{equation}\label{toprove}
\lim_{n\to+\infty}\PP_{\x_0}\big( \X_{k_1}=\x_1,\ldots, \X_{k_j}=\x_j\mid\N=\lf n\bv \rf\big)\\=\frac{\x_j\cdot\bu}{\x_0\cdot\bu}\PP_{\x_0}\big( \X_{k_1}=\x_1,\ldots, \X_{k_j}=\x_j\big).
\end{equation}
Let $\frac{3}{4}<\varepsilon<1$. The Markov property entails that
\begin{multline}\label{second term}
\PP_{\x_0}\left( \X_{k_1}=\x_1,\ldots, \X_{k_j}=\x_j\mid\N=\lf n\bv \rf\right)
\\=\sum_{\substack{\el\in\NN\\\el< \lfloor n^{\varepsilon}\bv \rfloor}} \PP_{\x_0}\left( \X_{k_1}=\x_1,\ldots, \X_{k_j}=\x_j,\N_{k_j}=\el\right)\frac{\PP_{\x_j}\left(\N=\lf n\bv \rf-\el+\x_j\right)}{\PP_{\x_0}\left( \N=\lf n\bv \rf\right)}\\\ \ +\sum_{\substack{\el\in\NN\\ \lfloor n^{\varepsilon}\bv \rfloor\leqslant\el\leqslant \lf n\bv \rf}} \PP_{\x_0}\left( \X_{k_1}=\x_1,\ldots, \X_{k_j}=\x_j,\N_{k_j}=\el\right)\frac{\PP_{\x_j}\left(\N=\lf n\bv \rf-\el+\x_j\right)}{\PP_{\x_0}\left( \N=\lf n\bv \rf\right)}.
\end{multline}
Note that \eqref{local} ensures that
\[\lim_{n} n^{\frac{d}{2}}\PP\left(\sum_{i=1}^d \ES^i_{\lf nv_i\rf -l_i+x_{j,i}}=\lf n\bv\rf -\el\right)=\frac{1 }{\left( 2\pi\right)  ^{\frac{d}{2}}\left(\det \mathbf{\Sigma}\right) ^{\frac{1}{2}}}, \]
uniformly in $\el< \lfloor n^{\varepsilon}\bv \rfloor$, and that the proof of Proposition \ref{prop: convergence} can be used to show that
\[
\lim_{n\to+\infty}n^{\frac{d}{2}+1}\PP_{\x_j }\left( \N=\lf n\bv \rf-\el+\x_j\right)=\frac{C\x_j\cdot \bu }{v_d\left( 2\pi\right)  ^{\frac{d}{2}}\left(\det \mathbf{\Sigma}\right) ^{\frac{1}{2}}},
\]uniformly in $\el< \lfloor n^{\varepsilon}\bv \rfloor$. Together with Proposition \ref{prop: convergence}, this shows that the first sum in \eqref{second term} converges to 
\begin{equation}\label{hop}
\frac{\x_j\cdot\bu}{\x_0\cdot\bu}\sum_{\el\in\NN} \PP_{\x_0}\big( \X_{k_1}=\x_1,\ldots, \X_{k_j}=\x_j,\N_{k_j}=\el\big)\\=\frac{\x_j\cdot\bu}{\x_0\cdot\bu}\PP_{\x_0}\big( \X_{k_1}=\x_1,\ldots, \X_{k_j}=\x_j\big)
\end{equation}
as $n\to+\infty$. The second sum in \eqref{second term} can be bounded by
\begin{align*}
\frac{\PP_{\x_j}\left(\N_{k_j}\geqslant \lf n^{\varepsilon}\bv\rf\right) }{\PP_{\x_0}\left( \N=\lf n\bv \rf\right)}&\leqslant
\frac{\PP_{\x_j}\left(\|\N_{k_j}\|_1^{d+1}\geqslant n^{(d+1)\varepsilon}\|\bv\|_1^{d+1}\right)}{\PP_{\x_0}\left( \N=\lf n\bv \rf\right)}\\
&\leqslant  \frac{\EE_{\x_j}\left(\|\N_{k_j}\|_1^{d+1}\right)}{\|\bv\|_1^{d+1}n^{(d+1)\varepsilon}\PP_{\x_0}\left( \N=\lf n\bv \rf\right)}.
\end{align*}
Thanks to $(A_5)$, the moments of order $d+1$ of the finite sum $\N_{k_j}$ are finite, and since $(d+1)\varepsilon>\frac{d}{2}+1$, the right term of the last inequality converges to 0 as $n\to +\infty$ thanks to Proposition \ref{prop: convergence}. This together with \eqref{hop} in \eqref{second term} finally proves \eqref{toprove}.

\end{document}